%% file: main.tex
\documentclass[conference]{ieeeconf}

\IEEEoverridecommandlockouts
\include{macros-abbrev}
\usepackage{url}
\usepackage{amsmath}
\usepackage{bbm}
\usepackage{color}

\usepackage{amsthm}
\usepackage{amssymb}
\usepackage{algpseudocode}
\usepackage{algorithm}
\algnewcommand{\Inputs}[1]{%
	\State \textbf{Input:}
	\Statex \hspace*{\algorithmicindent}\parbox[t]{.8\linewidth}{\raggedright #1}
}
\algnewcommand{\Initialize}[1]{%
	\State \textbf{Initialize:}
	\Statex \hspace*{\algorithmicindent}\parbox[t]{.8\linewidth}{\raggedright #1}
}
\algnewcommand{\Outputs}[1]{%
	\State \textbf{Output:}
	\Statex \hspace*{\algorithmicindent}\parbox[t]{.8\linewidth}{\raggedright #1}
}

\usepackage{relsize}
\usepackage{graphicx}
\graphicspath{ {images/} }

\newcommand{\longthmtitle}[1]{\mbox{}{\bf \textit{(#1).}}}
\theoremstyle{plain}

\newtheorem{lem}{Lemma}

\theoremstyle{definition}
\newtheorem{assump}{Assumption}

\theoremstyle{remark}

\DeclareMathOperator{\ones}{\mathbf{1}}
\DeclareMathOperator{\zeros}{\mathbf{0}}
\DeclareMathOperator{\I}{\mathbbm{I}}
\DeclareMathOperator{\G}{\mathcal{G}}
\DeclareMathOperator{\N}{\mathcal{N}}
\DeclareMathOperator{\R}{\mathbbm{R}}
\DeclareMathOperator{\E}{\mathcal{E}}

\DeclareMathOperator{\Ex}{\mathbbm{E}}
\DeclareMathOperator{\Pp}{\mathbbm{P}}

\DeclareMathOperator{\X}{\mathcal{X}}
\DeclareMathOperator{\Px}{\mathsf{P}}

\newcommand{\nll}{\operatorname{null}}
\newcommand{\aff}{\operatorname{aff}}
\newcommand{\relint}{\operatorname{relint}}

\makeatletter
\newcommand\RedeclareMathOperator{%
	\@ifstar{\def\rmo@s{m}\rmo@redeclare}{\def\rmo@s{o}\rmo@redeclare}%
}
\newcommand\rmo@redeclare[2]{%
	\begingroup \escapechar\m@ne\xdef\@gtempa{{\string#1}}\endgroup
	\expandafter\@ifundefined\@gtempa
	{\@latex@error{\noexpand#1undefined}\@ehc}%
	\relax
	\expandafter\rmo@declmathop\rmo@s{#1}{#2}}
\newcommand\rmo@declmathop[3]{%
	\DeclareRobustCommand{#2}{\qopname\newmcodes@#1{#3}}%
}
\@onlypreamble\RedeclareMathOperator
\makeatother

\RedeclareMathOperator{\S}{\mathcal{S}}
\RedeclareMathOperator{\P}{\mathcal{P}}
\renewcommand{\baselinestretch}{.99}

\begin{document}

\title{\huge Maximizing Algebraic Connectivity of Constrained Graphs in Adversarial Environments}

\author{Tor Anderson
\quad
Chin-Yao Chang
\quad
Sonia Mart\'{i}nez
\thanks{Tor Anderson, Chin-Yao Chang,  and Sonia Mart{\'\i}nez are
	with the Department of Mechanical and Aerospace Engineering,
	University of California, San Diego, CA, USA. Email: {\small {\tt
			\{tka001, chc433, soniamd\}@eng.ucsd.edu}}. This research was supported by the Advanced Research Projects Agency - Energy under the NODES program, Cooperative Agreement DE-AR0000695.}
}


\maketitle

\begin{abstract}
  This paper aims to maximize algebraic connectivity of networks via
  topology design under the presence of constraints and an
  adversary. We are concerned with three problems. First, we formulate
  the concave-maximization topology design problem of adding edges to
  an initial graph, which introduces a nonconvex binary decision
  variable, in addition to subjugation to general convex constraints
  on the feasible edge set. Unlike previous approaches, our method is
  justifiably not greedy and is capable of accommodating these
  additional constraints.  We also study a scenario in which a
  coordinator must selectively protect edges of the network from a
  chance of failure due to a physical disturbance or adversarial
  attack. The coordinator needs to strategically respond to the
  adversary's action without presupposed knowledge of the adversary's
  feasible attack actions. We propose three heuristic algorithms for
  the coordinator to accomplish the objective and identify worst-case
  preventive solutions.
  Each algorithm is shown to be effective in simulation and  their
  compared performance is discussed.
\end{abstract}

\IEEEpeerreviewmaketitle

\section{Introduction}
\textit{Motivation.} Multi-agent systems are pervasive in new
technology spaces such as power networks with distributed energy
resources like solar and wind, mobile sensor networks, and
large-scale distribution systems. In such systems, communication
amongst agents is paramount to the propagation of information, which
often lends itself to robustness and stability of the system. Network
connectivity is well studied from a graph-theoretic standpoint, but
the problem of designing topologies when confronted by engineering
constraints or adversarial attacks is not well addressed by current
works. We are motivated to study the NP-hard graph design problem of
adding edges to an initial topology and to develop a method to solve
it which has both improved performance and
allows for direct application to the aforementioned constrained and
adversarial settings.

\textit{Literature Review.} The classic paper~\cite{MF:73} by Miroslav
Fiedler proposes a scalar metric for the \emph{algebraic connectivity}
of undirected graphs, which is given by the second-smallest eigenvalue
of the graph Laplacian and is also referred to as the \emph{Fiedler}
eigenvalue. One of the main problems we are interested in studying is
posed in~\cite{AG-SB:06}, where the authors develop a heuristic for
strategically adding edges to an initial topology to maximize this
eigenvalue. Lower and upper bounds for the Fiedler eigenvalue with
respect to adding a particular edge are found; however, the work is
limited in that their approach is greedy and may not perform well in
some cases. In addition, the proposed strategy does not address how to
handle additional constraints that may be imposed on the network, such
as limits on nodal degree or restricting costlier edges. The authors
of~\cite{SB-AG-TB:13} aim to solve the problem of maximizing
connectivity for a particular robotic network scenario in the presence
of an adversarial jammer, although the work does not sufficiently
address scenarios with a more general adversary who may not be subject
to dynamical constraints. The Fiedler eigenvector, which has a close
relationship to the topology design problem, is studied
in~\cite{RM:98}. Many methods to compute this eigenvector exist, such
as the cascadic method in~\cite{JU-XH-JX-LZ:14}. However, these papers
do not fully characterize how this eigenvector evolves from adding or
removing edges from the network, which is largely unanswered by the
literature. The authors of~\cite{XD-TJ:10} study the spectra of
randomized graph Laplacians, and~\cite{PY-RAF-GJG-KML-SSS-RS:10} gives
a means to estimate and maximize the Fiedler eigenvalue in a mobile
sensor network setting. However, neither of these works consider the
problem from a design perspective. In the celebrated paper by Goemans
and Williamson~\cite{MG-DW:95}, the authors develop a relaxation and
performance guarantee on solving the MAXCUT problem, which has not yet
been adapted for solving the topology design problem. Each
of~\cite{NMMA:06,SB:06} survey existing results related to the Fiedler
eigenvalue and contain useful references.

\textit{Statement of Contributions.}  This paper considers three
optimization problems and has two main contributions. First, we
formulate the concave-maximization topology design problem from the
perspective of adding edges to an initial network, subject to general
convex constraints plus an intrinsic binary constraint. We then pose a
scenario where a coordinator must strategically select links to
protect from random failures due to a physical disturbance or
malicious attack by a strategic adversary. In addition, we formulate
this problem from the adversary's perspective. Our first main
contribution is a method to solve the topology design problem (and, by
extension, the protected links problem). We develop a novel
MAXCUT-inspired SDP relaxation to handle the binary constraint, which
elegantly considers the whole problem in a manner where previous
greedy methods fall short. Our next main contribution returns to the
coordinator-adversary scenario. We first discuss the nonexistence of a
Nash equilibrium in general. This motivates the development of an
optimal \emph{preventive} strategy in which the coordinator makes an
optimal play with respect to any possible response by the
adversary. We rigorously prove several auxiliary results about the
solutions of the adversary's computationally hard concave-minimization
problem in order to justify heuristic algorithms which
may be used by the coordinator to search for the optimal preventive
strategy. A desirable quality of these algorithms is they do not
presuppose the knowledge of the adversary's feasibility set, nor the
capability of solving her problem. Rather, the latter two algorithms
observe her plays over time and use these against her construct an
effective preventive solution. Simulations demonstrate the
effectiveness of our SDP relaxation for topology design and the
performance of the preventive-solution seeking algorithms when applied
to the  adversarial link-protection problem.


\section{Preliminaries}\label{sec:prelims}
This section establishes some notation and preliminary concepts which
will be drawn upon throughout the paper.

\subsection{Notation}\label{ssec:notation}
We denote by $\R$ the set of real numbers. The notation $x\in\R^n$ and $A\in\R^{n\times m}$ indicates an $n$-dimensional real vector and an $n$-by-$m$-dimensional real matrix, respectively. 
The gradient of a real-valued
multivariate function $f:\R^n\rightarrow\R$ with respect to a vector
$x\in\R^n$ is written $\nabla_x f(x)$. The $\supscr{i}{th}$ component
of a vector $x$ is indicated by $x_i$, all components of $x$ not including the $\supscr{i}{th}$ component is indicated by $x_{-i}$, and the $\supscr{(i,j)}{th}$ element of
a matrix $A$ is indicated by $A_{ij}$. 
The standard inner product is written 
$\left\langle x,y \right\rangle = x^\top y$ and the Euclidean norm of a
vector $x$ is denoted $\Vert x\Vert_2 = \sqrt{\langle x,x\rangle}$. The closed
Euclidean ball of radius $\epsilon$ centered at a point $x$ is expressed as
$\mathcal{B}_\epsilon (x)$. Two vectors $x$ and $y$
are perpendicular if $\left\langle x,y \right\rangle = 0$,
indicated by $x\perp y$, and the orthogonal complement to a span of vectors $a_i$ is written $\spn\{a_i\}^\perp$, meaning $x\perp y, \forall x\in\spn\{a_i\}, \forall y\in\spn\{a_i\}^\perp$. 
Elementwise multiplication is represented by $x \diamond y = (x_1 y_1, \dots
, x_n y_n)^\top$. A symmetric matrix $A\in\R^{n\times n}$ has $n$ real
eigenvalues ordered as $\lambda_1 \leq \dots \leq \lambda_n$, sometimes written $\lambda_i(A)$ if clarification is needed, with
associated eigenvectors $v_1, \dots , v_n$ that are assumed to be of
unit magnitude unless stated otherwise. A positive semi-definite matrix $A$ is indicated by $A \succeq 0$. We denote componentwise inequality as $x \succeq y$. The notation $\zeros_n$ and
$\ones_n$ refers to the $n$-dimensional vectors of all zeros and all
ones, respectively.  We use the notation $\I_n:= I_n -
\dfrac{\mathbf{1}_n\mathbf{1}^\top_n}{n}$ and refer to this matrix as
a \emph{pseudo-identity matrix}; note that $\nll{(\I_n)} = \spn
\{\ones_n\}$. 
The operator
$\operatorname{diag}$ for vector arguments produces a diagonal matrix
whose diagonal elements are the entries of the vector. The empty set
is denoted $\emptyset$ and the cardinality of a finite set
$\mathcal{S}$ is denoted $\vert\mathcal{S}\vert$. We express the $m$ Cartesian
product of sets by means of a superscript $m$, such as  
$[\underline{s}_i,\overline{s}_i]^m \subset \R^m$, where $\underline{s}_i,\overline{s}_i \in \R$. 
Probabilities and
expectations are indicated by $\Pp$ and $\Ex$, respectively. 

\subsection{Graph Theory}\label{ssec:graph-theory}

We refer to~\cite{CDG-GFR:01} as a supplement for the concepts we
describe throughout this subsection. In multi-agent engineering
applications, it is useful to represent a network mathematically as a
graph $\G = (\N , \E)$ whose node set is given by $\N = \until{n}$ and
edge set $\E \subseteq \N \times \N$, $\vert\E\vert = m$, which
represents a physical connection or ability to transmit a message
between agents.  We refer to the set of nodes that node $i$ is
connected to as its set of neighbors, $\N_i\subset\N$.  We consider
undirected graphs so $(i,j)\in\E$ indicates $j\in\N_i$ and $i\in\N_j$.
The graph has an associated Laplacian matrix $L\in\R^{n\times n}$,
whose elements are 
\begin{equation*}
L_{ij} =
\begin{cases}
  -1, & j \in \N_i, j \neq i, \\
  \vert \N_i \vert , & j = i, \\
  0, & \text{otherwise}.
\end{cases}
\end{equation*}

Note that $L \succeq 0$. It is  well known  that
the multiplicity of the zero eigenvalue is equal to the number of
connected components in the graph~\cite{CDG-GFR:01}. To expand on
this, \emph{connected} graphs have a one-dimensional null space
associated with the eigenvector $\ones_n$. The incidence matrix of $L$
is given by $E\in\{-1,0,1\}^{n\times m}$, where the $\supscr{l}{th}$
column of $E$, given by $e_l\in\{-1,0,1\}^{n}$, is associated with an
edge $l\sim(i,j)$. In particular, the $\supscr{i}{th}$ element of
$e_l$ is $-1$, the $\supscr{j}{th}$ element is $1$, and all other
elements are zero. A vector $x\in\{0,1\}^m$ encodes the
(dis)connectivity of the edges. In this sense, for
$l\in\until{m}$,~$x_l = 0$ indicates $l$ is disconnected and $x_l = 1$
indicates $l$ is connected. Then, $L = E \diag{x} E^\top$.

\subsection{Set Theory}

A limit point $p$ of a set $\P$ is a point such that any neighborhood
$\mathcal{B}_\epsilon (p)$ contains a point $p'\in\P$. A set is closed
if it contains all of its limit points, it is bounded if it is
contained in a ball of finite radius, and it is compact if it is both
closed and bounded. Let $\mathcal{A}_i = \setdef{p}{a_i^\top p \geq
  b_i}$ be a closed half-space and $\P = \mathcal{A}_1 \cap \dots \cap
\mathcal{A}_r \subset \R^m$ be a finite intersection of closed
half-spaces. If $\P$ is compact, we refer to it as a
polytope. Consider a set of points $\mathcal{F} =
\setdef{p\in\P}{a_i^\top p = b_i, i\in \mathcal{I} \subseteq
  \until{r}; a_j^\top p \geq b_j, j\in
  \until{r}\setminus\mathcal{I}}$. Let $h = \dim(\spn\{a_i\})$ be the
dimension of the subspace spanned by the vectors
$\{a_i\}_{i\in\mathcal{I}}$.  Then, we refer to $\mathcal{F}$ as an
$(m-h)$-dimensional face of $\P$. Lastly, denote the affine hull of
$\mathcal{F}$ as $\aff(\mathcal{F}) = \setdef{p + w}{p,w\in\R^m,
  p\in\mathcal{F}, w \perp \spn\{a_i\}_{i\in\mathcal{I}}}$ and define
the relative interior of $\mathcal{F}$ as $\relint(\mathcal{F}) =
\setdef{p}{\exists \epsilon > 0 :
  \mathcal{B}_\epsilon(p)\cap\aff(\mathcal{F}) \subset \mathcal{F}}$.

\section{Problem Statements}\label{sec:prob-formulation}

This section formulates the three optimization problems that we
study. The first problem aims to add edges to an initial topology to
maximize algebraic connectivity of the final graph. The second problem
introduces a coordinator who is charged with protecting some links in
a network that are subject to an external disturbance or
attack. 
The third problem takes the opposite approach of the latter problem by
minimizing connectivity from an adversarial point of view.

\subsection{Topology Design for Adding
  Edges}\label{ssec:top-des-add-edges}
Consider a network of agents with some initial (possibly
disconnected) graph topology characterized by an edge set $\E_0$ and
Laplacian
$L_0$. 
We would like to add $k$ edges to $\E_0$, possibly subject to some
additional convex constraints, so as to maximize the Fiedler
eigenvalue of the final Laplacian $L^\star$. 
This problem is well motivated: the Fiedler eigenvalue dictates convergence rate of many first-order distributed algorithms such as consensus and gradient descent.
First,
let $\E$ be the complete edge set (not including redundant edges or self loops) with $m = \vert\E\vert$. 
Consider the incidence matrix $E$ associated with $\E$ and the vector of edge connectivities $x$, as described in Section~\ref{ssec:graph-theory}. 
The constrained topology design problem is
then formulated as 
\begin{subequations} \label{eq:gen-prob}
	\begin{align}
	\Px 1: \
	& \underset{x,\alpha}{\text{max}}
	& & \alpha,  \label{eq:gen-fiedler-cost}\\
	& \text{subject to}
	& & E(\diag{x})E^\top \succeq \alpha\I_n , \label{eq:gen-fiedler-const}\\
	&&& \sum_{l=1}^m x_l \leq k+\vert\E_0\vert, \label{eq:gen-link-tot-const}\\
	&&& x\in \X, \label{eq:gen-cvx-const}\\
	&&& x_l = 1, \ l\sim (i,j) \in \E_0,\\
	&&& x_l \in \{0,1\}, \ l\in\until{m}. \label{eq:gen-link-const}
	\end{align}
\end{subequations}

In $\Px 1$, the solution $\alpha^\star$ is precisely the value for
$\lambda_2$ of the final Laplacian solution given by $L(x^\star) = E(\diag{x^\star})E^\top$. This is encoded in the
constraint~\eqref{eq:gen-fiedler-const}, where the pseudo-identity
matrix $\I_n$ has the effect of filtering out the fixed
zero-eigenvalue of the Laplacian. A useful relation is $\lambda_2 =
\underset{z}{\text{inf}} \setdef{z^\top L(x)z}{z\perp \ones_n, \Vert z
  \Vert_2 = 1}$, which shows that $\lambda_2$ as a function of $x$ is
a pointwise infimum of linear functions of $x$ and is therefore
concave. By extension, $\Px 1$ is a concave-maximization problem in
$x$. The set $\X$ is assumed compact and convex and may be chosen by
the designer in accordance with problem constraints such as
bandwidth/memory limitations, restrictions on nodal degrees, or
restricting certain edges from being 
chosen. These constraints may manifest in applications such as communication bandwidth limitations amongst Distributed Energy Resource Providers for Real-Time Optimization in renewable energy dispatch~\cite{CAISO-BPM:17}. 
The binary constraint~\eqref{eq:gen-link-const} is nonconvex and makes
the problem NP-hard. Handling this constraint is one of the main
objectives of this paper and will be addressed in
Section~\ref{sec:top-design}.

As for existing methods of solving $\Px 1$, one option is solve it
over the convex hull of the constraint set, which is given by $[0,1]^m
\cap \X$. Then, the problem may be solved in $k$ steps by iteratively
adding the edge $l\sim(i,j)$ for which $x_l$ is maximal in each
step. This is discussed in~\cite{AG-SB:06} and the references
therein. Although this method allows the designer to easily capture
$\X$, it is not a satisfying relaxation because the underlying
characteristics of the connectivity are not well captured. The authors
of~\cite{AG-SB:06} propose an alternate method which chooses the edge
$l$ for which $\dfrac{\partial \lambda_2}{\partial x_l} = v_2^\top
\dfrac{\partial L(x)}{\partial x_l}v_2 = v_2^\top e_l e_l^\top v_2 =
(v_{2,i} - v_{2,j})^2$ is maximal. This method is limited in that it
is (a) greedy and (b) cannot account for $\X$. We are motivated to
develop a relaxation for $\Px 1$ which improves on existing techniques
in both performance and the capability of handling
constraints. 

\subsection{Topology Design for Protecting Edges}\label{ssec:top-prot-nets}
We now formulate a problem which is closely related to $\Px 1$ and interesting to study in its own right. Motivated by the scenario of guarding against disruptive
physical disturbances or adversarial attacks, consider a coordinator
who may protect up to $k_s$ links from failing in a network. The
failure of the links are assumed to be independent Bernoulli random
variables whose probabilities are encoded by the vector
$p\in[0,1]^m$. Then, consider the coordinator's decision vector
$s\in\S = \{0,1\} \cap \S'$, where $\S'$ is assumed compact and convex. We write
out the Laplacian as before, $L(x) = E \diag{x} E^\top$. Following a
disturbance or attack, the probability that an edge $l$ is
(dis)connected is given by $\Pp (x_l \equiv 1) = (s_l-1)p_l + 1$
(resp. $\Pp (x_l \equiv 0) = (1-s_l)p_l$). The interpretation for the
vector $s$ here is that, if a particular element $s_l = 1$, it is
deterministically connected and considered \emph{immune} to the
disturbance or attack. The coordinator's problem is:
\begin{subequations} \label{eq:coord-prob}
	\begin{align}
	\Px 2: \
	& \underset{s,\alpha}{\text{max}}
	& & \alpha,  \label{eq:coord-fiedler-cost}\\
	& \text{subject to}
	& & \Ex\left[ E(\diag{x})E^\top \right] \succeq \alpha\I_n, \label{eq:coord-fiedler-const}\\
	&&& \Pp (x_l \equiv 1) = (s_l - 1)p_l + 1, \label{eq:coord-prbty-const1} \\
	&&& \Pp (x_l \equiv 0) = (1-s_l)p_l, \label{eq:coord-prbty-const2} \\
	&&& \sum_{l=1}^m s_l \leq k_s, \label{eq:coord-tot-const} \\
	&&& s\in\S, '
	\quad s\in \{0,1\}^m. \label{eq:coord-bin-const}
	\end{align}
\end{subequations}

Due to the linearity of the expectation operator,~\eqref{eq:coord-fiedler-const} is
equivalent to $E(\diag{(s - \ones_m)\diamond p + \ones_m})E^\top
\succeq \alpha \I_n$, which is an LMI (linear matrix inequality) in $s$. We note that, as in $\Px
1$, the objective of $\Px 2$ may be thought of as a pointwise infimum
of linear functions and, as such, is a binary concave-maximization
problem in $s$.

The formulation in $\Px 2$ presupposes a fixed vector $p$. However, it
may be the case that a strategic attacker detects preventive action
taken by the coordinator and adjusts her strategy to improve the
likelihood of disconnecting the network. We now formulate the
attacker's problem for some known, fixed coordinator strategy $s$:
\begin{subequations} \label{eq:atk-prob}
	\begin{align}
	\Px 3: \
	& \underset{p,\alpha}{\text{min}}
	& & \alpha, \label{eq:atk-fiedler-cost}\\
	& \text{subject to}
	& & \lambda_2 (\Ex\left[ E(\diag{x})E^\top \right]) = \alpha, \label{eq:atk-fiedler-const}\\
	&&& \Pp (x_l \equiv 1) = (s_l - 1)p_l + 1, \label{eq:atk-prbty-const1} \\
	&&& \Pp (x_l \equiv 0) = (1-s_l)p_l, \label{eq:atk-prbty-const2} \\
	&&& p\in\P, \quad p\in [0,1]^m. \label{eq:atk-cvx-const}
	\end{align}
\end{subequations}

Notice here that the optimization is instead over $p$, and is now a
minimization of $\alpha$. The constraint~\eqref{eq:atk-fiedler-const}
is now a nonlinear equality rather than an LMI, which manifests itself from this being a
concave-\emph{minimization} problem. This equality is not a convex constraint 
and will be addressed in Section~\ref{sec:prot-links}. We assume $\P$
is compact and convex.

\section{An SDP Relaxation for Topology Design}\label{sec:top-design}

This section aims to develop a relaxed approach to solve $\Px 1$ in a computationally efficient manner. Ideally, such
an approach may also be straightforwardly extended to problems of the
form $\Px 2$. To do this, we draw on intuition from the randomized
hyperplane strategy given in~\cite{MG-DW:95} for solving the
well-studied MAXCUT problem.

There are two notable differences between $\Px 1$ and MAXCUT: the
entries of the decision vector in $\Px 1$ take values in $\{0,1\}$,
whereas in MAXCUT, the decision (let's say $z$) takes values $z_i \in
\{-1,1\}$. The latter is convenient because it is equivalent to $z_i^2
= 1$. Additionally, the enumeration in MAXCUT is \emph{symmetric} in
the sense that, if $z^\star$ is a solution, then so is
$-z^\star$. However, $\Px 1$ is \emph{assymmetric} in the sense that,
if $x^\star$ is a solution, it \emph{cannot} be said that $-2x^\star +
\ones_m$ (effectively swapping the zeros and ones in the elements of
$x^\star$) is a solution. We rectify these issues with a
transformation and variable lift, respectively. Introduce a vector $y
= 2x - \ones_m$
and notice $x\in\{0,1\}^m$ maps to $y\in\{-1,1\}^m$. Then, define $Y =
yy^\top$ so that $y_l^2 = 1$ may be enforced via $Y_{ll} = 1$,
$l\in\until{m}$.
In addition, define $\widetilde{Y} = \begin{bmatrix}y \\ 1\end{bmatrix}\begin{bmatrix}y \\
  1\end{bmatrix}^\top = \begin{bmatrix}Y & y \\ y^\top &
  1\end{bmatrix}$ to capture the
asymmetry 
in the original variable $x$. Now, we are ready to reformulate $\Px 1$
as an SDP in the variable $y$:
\begin{subequations} \label{eq:SDP-prob}
	\begin{align}
	\Px 4: \
	& \underset{Y,y,\alpha}{\text{max}}
	& & \alpha,  \label{eq:SDP-fiedler-cost}\\
	& \text{subject to}
	& & \dfrac{1}{2}E(\diag{y} + I_m)E^\top \succeq \alpha\I_n, \label{eq:SDP-fiedler-const}\\
	&&& \widetilde{Y} =	\begin{bmatrix}Y & y \\ y^\top & 1\end{bmatrix} \succeq 0,  \label{eq:SDP-bigY-const}\\
	&&& \rank{(\widetilde{Y})} = 1, \label{eq:SDP-rank-const}\\
	&&& \widetilde{Y}_{ll} = 1, \ l\in\until{m},\label{eq:SDP-ii-const}\\
	&&& y\in\mathcal{Y}, \label{eq:SDP-cvx-const}\\
	&&& y_l = 1, \ l\sim (i,j)\in \E_0,\\
	&&& \dfrac{1}{2}\sum_i (y_i + 1) \leq k + \vert\E_0\vert, \label{eq:SDP-link-tot-const}
	\end{align}
\end{subequations}
where $\mathcal{Y}$ is an affine transformation on the set
$\mathcal{X}$ in~\eqref{eq:gen-cvx-const}, and we have simply used the
transformation and variable lift 
to rewrite the other constraints. The problem $\Px 4$ is
equivalent to $\Px 1$: the NP-hardness now manifests itself in the
nonlinear constraint~\eqref{eq:SDP-rank-const}. Dropping this
constraint produces a \emph{relaxed} solution $\widetilde{Y}^\star$
with the rank of $\widetilde{Y}^\star$ not necessarily
one. 

This also produces a solution $y^\star$ which can be mapped back to
$x^\star$. Of course, $x^\star$ may not take binary values due to the
dropped rank constraint. We now briefly recall the geometric intuition
to approximate the solution to MAXCUT in~\cite{MG-DW:95} with many
technical details omitted here for brevity. Let $Z^\star\in\R^{m_z
  \times m_z}$ be a rank $r_z$ solution to the rank-relaxed MAXCUT SDP
problem. Decompose $Z^\star = W^\top W$ with $W \in\R^{r_z\times
  m_z}$, and notice the columns of $W$ given by $w_i \in \R^{r_z}$,
$i\in\until{m_z}$, are vectors on the $r_z$-dimensional unit ball due
to $Z^\star_{ii} = 1$, $i\in\until{m_z}$. In order to determine a solution $z^\star\in\{-1,1\}^{m_z}$, generate a uniformly
random unit vector $p\in\R^{r_z}$ which may define a
hyperplane. 
If the vector $w_i$ lies
on one side of the hyperplane, i.e. $\langle w_i, p\rangle \geq 0$,
set the corresponding element $z_i = 1$. If it is on
the other side of the hyperplane, $\langle w_i, p\rangle < 0$, set $z_i = -1$. 
Geometrically speaking,
the stronger a vector $w_i$ is aligned with $p$, the more
``correlated'' (for lack of a better term) node $i$ is with the set $\setdef{j\in\until{m_z}}{z_j = 1}$
and vice-versa. From another perspective, consider the case
$r_z = 1$ which implies the solution is equivalent to the nonrelaxed
problem. Then, it can be seen that $\langle w_i, p\rangle \in
\{-1,1\}$ and the approach gives the exact optimizer for MAXCUT.

For our problem, decompose $\widetilde{Y}^\star = U^\top U$ with $U\in
\R^{r\times m+1}$, and obtain unit-vectors $u_l\in\R^r$,
$l\in\until{m+1}$, from the columns of $U$. Because of the asymmetry
of our problem, we do not implement a random approach to determine the
solution. Instead, notice that the last column $u_{m+1}$ is
qualitatively different than $u_l$, $l\in\until{m}$ due to the
variable lift. 
We have
that 
$y_l = \langle u_l ,
u_{m+1}\rangle, l\in\until{m}$. Thus, larger entries of $y_l$ correspond to vectors
$u_l$ on the unit ball which are more ``aligned'' with $u_{m+1}$,
which hearkens to the geometric intuition for the MAXCUT
solution and $u_{m+1}$ may be thought of as the vector $p$ from MAXCUT. The entries $y_l$ give a quantitative measure of the effectiveness of adding edge $l\sim (i,j)$. 

We suggest iteratively choosing the edge $l$ associated with the largest element of $y$ for
which $l\notin \E_0$. 
If a particular edge is infeasible, this is
elegantly accounted for by~\eqref{eq:SDP-cvx-const} and is reflected in the relaxed solution to $\Px 4$. This
approach may be iterated $k$ times, updating $\E_0$ and decrementing
$k$ each time in accordance with~\eqref{eq:SDP-link-tot-const} to construct a satisfactory solution to the original
NP-hard binary problem. In addition, this formulation is easily
adaptable to solve $\Px 2$ via a similar transformation and variable
lift in $s$. 

\section{Protecting Links Against an Adversary}\label{sec:prot-links}

This section begins by studying the Nash equilibria of a game between
the coordinator and attacker where they take turns solving $\Px 2$ and
$\Px 3$. We first study the (non)existence of the Nash equilibria of
this game, and use this result to motivate the development of a
\emph{preventive} strategy for the coordinator. We then provide some
auxiliary results about the solutions of $\Px 3$ and use these to
justify methods for finding such a preventive strategy.

\subsection{Nash Equilibria}

We begin by adopting the shorthand notation $L(s,p) = \Ex\left[
  E(\diag{x})E^\top \right]$ with $x$ distributed as
in~\eqref{eq:coord-prbty-const1}--\eqref{eq:coord-prbty-const2},
i.e. $L_{ij}(s,p) = (1-s_l)p_l-1, l\sim(i,j)\in\E, i\neq j$ and
$L_{ii}(s,p) = -\sum_{j\in\N_i}
L(s,p)_{ij}$. 
This matrix may be
interpreted 
as a weighted Laplacian whose elements are given by the righthand side
of~\eqref{eq:coord-prbty-const1}, \eqref{eq:atk-prbty-const1}. We also
adopt the shorthand $\alpha(s,p) =
\underset{z}{\text{inf}}\setdef{z^\top L(s,p) z}{z \perp \ones_n ,
  \Vert z \Vert_2 = 1}$ to refer to the Fiedler eigenvalue of
$L(s,p)$, and note that $\alpha (s,p)$, as a pointwise infimum of
bilinear functions, is concave-concave in $(s,p)$. From this, recall
the first-order concavity relation~\cite{SB-LV:04}
\begin{equation}\label{eq:alpha-concave}
	\begin{aligned}
          \alpha(s^2,p) &\leq \alpha (s^1,p) + \nabla_s \alpha (s^1,p)^\top (s^2 - s^1), \ \forall s^1, s^2 \in\S,\\
          \alpha(s,p^2) &\leq \alpha (s,p^1) + \nabla_p \alpha (s,p^1)^\top (p^2 - p^1), \ \forall p^1, p^2\in\P. \\
	\end{aligned}
\end{equation}

To get a better grasp on this, we compute the gradient of $\alpha$
with respect to both $s$ and $p$. Let $v$ be the Fiedler eigenvector
associated with the second-smallest eigenvalue (in this case,
$\alpha$) of $L(s,p)$. Then, 
\begin{equation}\label{eq:partial-s-alpha}
  \dfrac{\partial \alpha}{\partial s_l} = v^\top \dfrac{\partial
    L(s,p)}{\partial s_l} v ,
  = v^\top p_l e_l e_l^\top v = p_l(v_i - v_j)^2,
\end{equation}
which is a straightforward extension of the computation shown near the
end of Section~\ref{ssec:top-des-add-edges}. Additionally, 
\begin{equation}
\label{eq:partial-p-alpha}
\dfrac{\partial \alpha}{\partial p_l} = \begin{cases}
-(v_i - v_j)^2, & s_l = 0\\
0, & s_l = 1.
\end{cases}
\end{equation}
The gradient with respect to $s$ and $p$ is a vector with 
elements given by~\eqref{eq:partial-s-alpha}--\eqref{eq:partial-p-alpha},
which are nonnegative for $s$ and nonpositive for $p$. Also,
note that $v \neq \ones_n$, $v \neq 0$, implying the quantity $(v_i -
v_j)^2$ must be strictly positive for some edges $l\sim (i,j)$.

Now, consider a game where the coordinator and attacker take turns
solving and implementing the solutions of $\Px 2$ and $\Px 3$,
respectively. A Nash equilibrium is a point $(s^\star,p^\star)$ with
the property
\begin{equation}\label{eq:NE}
  \alpha(s,p^\star) \leq \alpha (s^\star,p^\star) \leq \alpha (s^\star,p), \forall s\in \S, \forall p\in\P,
\end{equation}
which is a stationary point of the aforementioned game. We now state a
lemma 
to motivate the remainder of this section.

\begin{lem}\longthmtitle{Nonexistence of Nash
    Equilibrium}\label{lem:no-nash} {\rm A Nash equilibrium point
    $(s^\star,p^\star)$ satisfying~\eqref{eq:NE} is not guaranteed to
    exist in general.}
\end{lem}

\begin{proof}
	To show this result, we provide a simple counterexample. Consider a complete graph with $n=3$ nodes and $m=3$ edges, $l\in\{1,2,3\} \sim \{(1,2),(1,3),(2,3)\} = \E$, and the coordinator and attacker feasibility sets $\S = \{0,1\}^3 \cap \setdef{s}{\sum_l s_l \leq 2}$, $\P = \setdef{p}{\sum_l p \leq 1}$. We consider three cases of the strategies $(s,p)$ to show the nonexistence of a Nash equilibrium here.
	
	\textbf{Case 1:} The attacker play $p^\star$ is such that $p^\star_l = 1$ for $l$ such that $s^\star_l = 1$. This point violates the righthand side of~\eqref{eq:NE}. To see this, notice $\exists l'\neq l$ s.t. $s^\star_{l'} = 0$ and that the play $p^{\star\prime}$ with $p^{\star\prime}_{l'} = 1$ is optimal for $\Px 3$.
	
	\textbf{Case 2:} The attacker play $p^\star_l$ is such that $p^\star_l = 1$ for $l$ such that $s^\star_l = 0$. This point violates the lefthand side of~\eqref{eq:NE}, and the coordinator recomputes an optimal play $s^{\star\prime}$ with $s^{\star\prime}_{l} = 1$.
	
	\textbf{Case 3:} The attacker play $p^\star$ is such that $\nexists l$ with $p^\star_l = 1$. This point violates the righthandside of~\eqref{eq:NE} and the attacker recomputes an optimal play $p^{\star\prime}$ with $p^{\star\prime}_l = 1$ for the $l$ corresponding to $s^\star_l = 0$.
\end{proof}

The simple counterexample employed in the proof of
Lemma~\ref{lem:no-nash} suggests that Nash equilibria are also
unlikely to exist in more meaningful scenarios. This result should not
come as a surprise: the solutions to $\Px 2$ and $\Px 3$ are in direct
conflict with one another, and playing sequentially has the effect of
the coordinator ``chasing'' the attacker around the network. We find
that the cases for which we can construct a Nash equilibrium are
trivial: for example, if $\ones_m \in \S$, the coordinator may choose
$s^\star = \ones_m$, and the attacker's solution set is trivially the
whole set $\P$ with the interpretation that the attacker is powerless
to affect the value of $\alpha$. Then, $(\ones_m,p)$ are Nash
equilibria $\forall p\in\P$, and are not interesting.

\subsection{Coordinator's Preventive Strategy}\label{ssec:prev-strat}

Lemma~\ref{lem:no-nash} motivates the study of an optimal
\emph{preventive} strategy for the coordinator under the assumption
that the attacker may always make a play in response to the
coordinator's action. Instead of a Nash equilibrium satisfying~\eqref{eq:NE}, we
seek a point $(s^\star,p^\star)$ satisfying the following:
\begin{equation}\label{eq:prevent-strat}
  (s^\star,p^\star) = \underset{s\in\S}{\text{argmax}} \ 
  \underset{p\in\P}{\text{argmin}} \ \alpha(s,p).
\end{equation}
The interpretation of $s^\star$ solving~\eqref{eq:prevent-strat} is
that it provides the best-case solution for the coordinator given that
the attacker makes the last play. In this sense, $s^\star$ is not
optimal for $p^\star$; rather, it is an optimal play with respect to
the whole set $\P$.

From the coordinator's perspective, the objective function
${\underset{p\in\P}{\text{argmin}} \ \alpha(s,p)}$ is a pointwise
infimum of concave functions of $s$, and therefore the problem is a
concave maximization. However, computing such a point may be dubious
in practice, particularly since we have not assumed the coordinator
has the capability of solving the concave minimization problem $\Px 3$
or even knowledge of $\P$. It would, however, be convenient to use the
attacker's solutions to $\Px 3$ against herself. To do this, we
establish some lemmas to gain insight on the solution sets of the
attacker. This helps us construct heuristics for computing $s^\star$
in the sense of~\eqref{eq:prevent-strat}. We assume $\S, \P \neq
\emptyset$.

\begin{lem}\longthmtitle{Attacker's Solution Tends to be Noninterior}\label{lem:nonint-point} {\rm Consider the set of solutions
    $\P^\star \subseteq \P$ to $\Px 3$ for some $s$. If there exists
    point $p^\star \in \P^\star$ which is an interior point of $\P$,
    then $\P^\star = \P$.}
\end{lem} 
\begin{proof}
  Consider $p^\star \in \P^\star$ which is an interior point of
  $\P$. Pick $\epsilon > 0$ so $\mathcal{B}_\epsilon(p^\star)
  \subseteq \P$. From~\eqref{eq:alpha-concave}, the point $p' =
  p^\star - \epsilon\nabla_p \alpha (s,p^\star)/\Vert \nabla_p \alpha
  (s,p^\star)\Vert_2 \in\P$ violates the condition that $p^\star$ is a
  solution to $\Px 3$ unless $\nabla_p \alpha(s,p^\star) = \zeros_m$,
  so this must be the case. Then, in accordance
  with~\eqref{eq:alpha-concave}, all $p \in \P$ are solutions, and
  the statement $\P^\star = \P$ follows.
\end{proof}

This lemma implies the solution set $\P^\star$ consists of noninterior
point(s) of $\P$ except in trivial cases. We now provide a stronger
result in the case where $\P$ is a polytope, which shows that
solutions tend to be contained in low-dimensional faces of $\P$ such
as edges (line segments) and vertices.

\begin{lem}\longthmtitle{Attacker's Solutions Tend Towards Low-Dimensional Faces}\label{lem:low-dim-face} 
  {\rm Let $\P = \mathcal{A}_1 \cap \dots \cap \mathcal{A}_r \subset
    \R^m$ be a compact polytope with half-spaces $\mathcal{A}_i$
    characterized by $a_i,b_i$ for $i\in\until{r}$, 
    and let
    $\mathcal{F}$ be a face of $\P$ with $a_j^\top p = b_j$ for
    $ j\in\mathcal{J}\subseteq \until{r}$,
    $\forall p\in\mathcal{F}$. 
    If a
    point $p^\star\in\relint{(\mathcal{F})}$ is 
    a solution to $\Px
    3$, then $\nabla_p \alpha(s,p^\star) \in \spn\{a_j\}_{j\in\mathcal{J}}$.}

\end{lem}
\begin{proof}
  Suppose that $\nabla_p \alpha(s,p^{\star}) \notin
  \spn\{a_j\}_{j\in\mathcal{J}}$ and decompose $\nabla_p
  \alpha(s,p^{\star}) = \beta+\zeta, \beta \in \spn\{a_j\}_{j\in\mathcal{J}}, \zeta\in
  \spn\{a_j\}_{j\in\mathcal{J}}^\perp, \zeta\neq \zeros_m$. 
Now consider the point
  $p' = p^{\star} - \epsilon \zeta/\Vert \zeta \Vert_2
  \in \mathcal{F}$ for some $\epsilon >
  0$. From~\eqref{eq:alpha-concave}, we have the relation
  $\alpha(s,p') \leq \alpha(s,p^{\star}) - \epsilon \Vert
  \zeta \Vert_2$ with $\epsilon > 0$ and $\Vert \zeta \Vert_2 > 0$, which contradicts
  $p^{\star}$ being a solution and completes the proof.
\end{proof}

To interpret the result of Lemma~\ref{lem:low-dim-face}, notice that $p\in\relint{(\mathcal{F})}$ implies $p$ does not belong to a lower dimensional face, and that the
the dimension of $\spn\{a_j\}_{j\in\mathcal{J}}$ 
becomes large only as the dimension of $\mathcal{F}$ becomes small. This intuitively suggests that the gradient of $\alpha$ at $p^\star$ may only
belong to $\spn\{a_j\}_{j\in\mathcal{J}}$ for a $p\in\mathcal{F}$ if this span is large
in dimension. This allows us to conservatively characterize the
solution set $\P^\star$, and the result gives credence to the notion
that solutions take values in low-dimensional faces of $\P$.

It is our intent to 
use
Lemmas~\ref{lem:nonint-point} and~\ref{lem:low-dim-face} to establish
intuition for the problem and justify solution strategies to the hard
problem of computing a preventive $s^\star$. Before proceeding, we establish one more simple lemma and provide
discussion on deterministically connected graphs.

\begin{lem}\longthmtitle{Deterministic Connectivity}\label{lem:det-conn}
  {\rm Assume $\exists p \in\P$ with $p \succ \zeros_m$ and that the attacker
    makes the last play. Let $L(x)$ indicate the random matrix whose
    elements are distributed as
    in~\eqref{eq:atk-prbty-const1}--\eqref{eq:atk-prbty-const2}. Then,
    a coordinator's strategy $s$ gives $\lambda_2 > 0$ of $L(x)$ with
    probability $1$ if and only if the elements of $s$ equal to $1$
    are associated with edges of a connected graph.}
\end{lem}
\begin{proof}
  $\Rightarrow)$ We do not need to assume the attacker has made an
  optimal play with respect to $\Px 3$. Instead, consider any play $p
  \succ \zeros_m$ and note that each edge $l$ associated with $s_l = 0$ has a
  nonzero probability of being disconnected. Then, there is a nonzero
  probability that $x_l = 0$ for each $l$ with $s_l = 0$, and the
  remaining protected edges do not form a connected graph.  Then, if
  the elements of $s$ equal to one are not associated edges defining a
  connected graph, $\lambda_2 = 0$ is an event with nonzero
  probability. 
	
  $\Leftarrow)$ This direction is trivial: $x_l = 1$ with probability
  $1$ for edges $l$ corresponding to $s_l = 1$. If these edges form a
  connected graph, then $\lambda_2 > 0$ with probability $1$.
\end{proof}

The consequence of Lemma~\ref{lem:det-conn} is obvious: if the
coordinator does not have the resources to protect edges which form a
connected graph and the attacker targets all edges, then there is no
guarantee the resulting graph will be connected. A subject of future
work is to provide some insight on the lower bound of $\lambda_2$ for
particular cases of $\S$ and $\P$.

\subsection{Heuristics for Computing a Preventive Strategy}

Recall from the previous subsection that the goal is to compute
$s^\star$ as a solution to~\eqref{eq:prevent-strat}. In this
subsection, we describe three approaches to computing a satisfactory
solution and formally adopt the following assumptions.
\begin{assump}\longthmtitle{Coordinator's Problem is
    Solvable}\label{assump:coord-solvable}
  {\rm Given a known vector $p$, the coordinator can find the optimal
    solution of $\Px 2$.}
\end{assump}
\begin{assump}\longthmtitle{Attacker Plays Optimally and Last}\label{assump:atk-optimal}
  {\rm The attacker's play $p$ belongs to a convex, compact set
    $\P$. In addition, she always makes optimal plays which solve $\Px
    3$ given the coordinator's play $s$, and she may always play in
    response to the coordinator changing his decision.}
\end{assump}
\begin{assump}\longthmtitle{Available Information}\label{assump:coord-info}
  {\rm The coordinator is cognizant of
    Assumption~\ref{assump:atk-optimal} and has access to the current
    attacker play $p$. He may compute $\alpha(s,p)$ for a particular
    play $(s,p)$.}
\end{assump}
\begin{assump}\longthmtitle{Only Last Play Matters}\label{assump:payout}
  {\rm The objective $\alpha(s,p)$ (and, by extension, $\lambda_2$) is
    only consequential once both the coordinator and attacker have
    chosen their final strategy and do not make additional plays.}
\end{assump}

We now construct three heuristics for computing
$s^\star$, the latter two of which observe the attacker plays $p(t)$
over a time horizon $t\in\until{T}$ and iteratively construct a solution.

\begin{algorithm}
  \caption{Random Sampling}\label{alg:rand}
  \begin{algorithmic}[1]
    \Procedure{Rand}{$\S, \E, T$}
     \State Compute $(v_i - v_j)^2$ for each $l\sim(i,j)\in\E$
    \For{$t=1,\dots,T$} \State $\mathcal{M} \gets \until{m}$ \State
    $s(t) \gets \zeros_m$ \State done $\gets false$ \While{done $=
      false$} \State Choose $l\in\mathcal{M}$ uniformly
    randomly 
    \If{$s_l(t)\gets 1 \Rightarrow s(t)\in\S$} \State $s_l(t) \gets 1$
    \EndIf
    \State $\mathcal{M} \gets \mathcal{M}
    \setminus \{l\}$ \If{$\nexists
      l\in\mathcal{M}$ s.t. $s_l(t) \gets
      1 \Rightarrow s(t) \in\S$} \State
    done $= true$
    \EndIf
    \EndWhile
    \State Play $s(t)$
    \State Store $s(t)$ and $\alpha(s(t),p(t))$
    \EndFor
    \State $t^\star \gets
    \underset{t}{\text{argmax}} \
    \alpha(s(t),p(t))$ \State \textbf{return}
    $s^\star = s(t^\star)$
    \EndProcedure
  \end{algorithmic}
\end{algorithm}
Algorithm~\ref{alg:rand} is simple and doesn't utilize the attacker's
plays. For each $t$, it constructs $s(t)$ by picking edges uniformly
randomly. 
Each loop terminates when no feasible edges remain. The value of
$\alpha$ is recorded following the attacker's response $p(t)$, and the
best $s(t)$ is returned.

\begin{algorithm}
  \caption{Convex Combinations}\label{alg:cvx-comb}
  \begin{algorithmic}[1]
  	\Require Convex weighting function $\eta$ with $\sum_{k=1}^t \eta(k) = 1$
    \Procedure{Cvx}{$\S, \E, T$}
    \State $s(1) \gets \zeros_m$
    \State Store $s(1)$, $p(1)$ and $\alpha(s(1),p(1))$
    \State $p_{\theta}(1) \gets p(1)$
    \For{$t=2,\dots,T$}
    \State $s(t) \gets \underset{s\in\S}{\text{argmax}} \ \alpha (s,p_{\theta}(t-1))$
    \State Store $s(t)$, $p(t)$, and $\alpha(s(t),p(t))$
    \For{$k=1,\dots,t}$
    \State $\theta(k) \gets \eta (k)$
    \EndFor
    \State $p_{\theta}(t) \gets \sum_{k=1}^t \theta(k) p(k)$
    \EndFor
    \State $t^\star \gets \underset{t}{\text{argmax}} \ \alpha(s(t),p(t))$
    \State \textbf{return} $s^\star = s(t^\star)$
    \EndProcedure
  \end{algorithmic}
\end{algorithm}
Algorithm~\ref{alg:cvx-comb} utilizes a convex weighting function
$\eta$ such that $\sum_{k=1}^t\hspace{-1mm} \eta(k) \hspace{-1mm}=
\hspace{-1mm}1$. We suggest three possible choices for $\eta$:
\begin{equation*}
\begin{aligned}
\eta_1(k) &= 1/t, \\
\eta_2(k) &= \dfrac{\gamma^{(t-k)}}{\sum_k \gamma^{(t-k)}}, \ \gamma\in(0,1), \\
\eta_3(k) &= \dfrac{\alpha(s(k),p(k))}{\sum_k \alpha(s(k),p(k))}.
\end{aligned}
\end{equation*}
These may be interpreted 
as a uniform weighting of each
observation $p(t)$, a recency-biased weighting, and a penalty-biased
weighting, respectively. Algorithm~\ref{alg:cvx-comb} is motivated by a few
observations. Firstly, recall
Lemmas~\ref{lem:no-nash}--\ref{lem:low-dim-face} and note that $p(t)$
may jump around extreme points of $\P$ as $s(t)$ evolves. 
Successive convex combinations of the solutions $p(t)$ effectively push the coordinator's decision towards responding to the vulnerable parts of
the space over time. 

Algorithm~\ref{alg:point} adopts the following stronger version of
Assumption~\ref{assump:coord-solvable}.
\begin{assump}\longthmtitle{Coordinator's Problem is Solvable Over Finitely-Many Points}
  {\rm Given a finite set of points $\overline{\P} \subset
    \P$, the coordinator may compute the solution ${s^\star = \underset{s\in\S}{\text{argmax}} \ \underset{p\in\overline{\P}}{\text{min}} \ \alpha(s,p)}$.
    }
\end{assump}

\begin{algorithm}
  \caption{Constructing $\P$ via Pointwise Search}\label{alg:point}
  \begin{algorithmic}[1]
    \Procedure{Search}{$\S, \E, T$}
    \State $s(1) \gets \zeros_m$
    \State Store $s(1)$, $p(1)$, and $\alpha (s(1),p(1))$
    \State $\overline \P \gets p(1)$
    \For{$t=2,\dots,T$}
    \State $s(t) \gets \underset{s\in\S}{\text{argmax}} \ \underset{p\in\overline{\P}}{\text{min}} \ \alpha(s,p)$
    \State Store $s(t)$, $p(t)$ and $\alpha (s(t),p(t))$
    \State $\overline{\P} \gets \overline{\P} \cup \{p(t)\}$
    \EndFor
    \State $t^\star \gets \underset{t}{\text{argmax}} \ \alpha(s(t),p(t))$
    \State \textbf{return} $s^\star = s(t^\star)$
    \EndProcedure
  \end{algorithmic}
\end{algorithm}

Algorithm~\ref{alg:point} operates by computing the solution $s(t)$ as
the optimal play with respect to each of the previous attacker plays
$p(k), k < t$. Although it is more computationally demanding than
Algorithms~\ref{alg:rand} and~\ref{alg:cvx-comb}, it is strongly
rooted in the theoretical understanding of the problem we have developed in the
following sense: the convex hull of these points
$\chull{\overline{\P}}$ at time $t$ is a compact polytope whose
vertices are defined by the points $p(k)$. Applying
Lemma~\ref{lem:low-dim-face}, it stands to reason that points in the
interior or in higher-dimensional faces of
$\chull{\overline{\P}}$ 
are uncommon solutions. We expect $\chull{\overline{\P}}$ to grow in each loop of
the algorithm and effectively reconstruct the attacker's feasibility
set $\P$. 
For
now, convergence to the true $s^\star$ is not guaranteed due to the
difficulty of characterizing the evolution of $\nabla_p
\alpha(s,p)$ with $p$. However, we note in simulation studies that
Algorithm~\ref{alg:point} converges to the global optimizer in a few
iterations. Finally, we state the following trivial lemma for completeness.
\begin{lem}\longthmtitle{Nondecreasing Performance of Algorithms~\ref{alg:rand}--\ref{alg:point}}
  {\rm Let $p^\star$ solve $\Px 3$ for $s = s^\star(T)$, where
    $s^\star(T)$ is the returned strategy of
    Algorithm~\ref{alg:rand},~\ref{alg:cvx-comb}, or~\ref{alg:point}
    truncated at time $T$. For all $T > 1$, $\alpha(s^\star(T),
    p^\star) \geq \alpha(s^\star(T-1),p^\star)$.}
\end{lem}
\begin{proof}
  The result is trivially seen in that $s^\star(T) = s(t^\star),
  t^\star = \underset{t\in\until{T}}{\text{argmax}} \
  \alpha(s(t),p(t))$ and $\until{T-1} \subset \until{T}$.
\end{proof}

\section{Simulations}\label{sec:sims}

We now examine our proposed SDP relaxation for solving $\Px 1$. For
the ease of comparison with the Fiedler vector heuristic given
in~\cite{AG-SB:06}, we do not include any additional convex
constraints beyond~\eqref{eq:gen-link-tot-const}. For a network with
$14$ nodes, $\vert \E_0 \vert = 28$ initial edges generated randomly,
we implement the Fiedler method, our SDP method, and the approach of
taking the convex hull of the feasibility set of $\Px 1$. We compute
$100$ trials with different topology initializations for each of the
$k\in\{25,40\}$ edge-addition cases. Our SDP design outperforms the
Fiedler vector heuristic in $75$ trials for the $k=25$ case, $80$
trials for the $k=40$ case, and it outperforms the convex hull
approach in all $100$ trials of both cases. This improved performance
is observed over a variety of network sizes and initial
connectivities, and we observe that increasing $k$ increases the
likelihood that our method outperforms the alternatives.

One such instance of $k = 25$ added edges is plotted in
Figure~\ref{fig:net-tops} and the performance is plotted in
Figure~\ref{fig:top-des}. Firstly, note that both our SDP method and
the Fiedler vector method greatly outperform the simple convex hull
approach. 
Our SDP method is outperformed by the Fiedler method in early
iterations. 
In later iterations, the performance of our
method catches up with and surpasses the Fiedler vector heuristic
(this is common behavior across other initializations). We contend
that the reason for this is that the relaxed solution of $\Px 4$ at each
iteration is cognizant of the entire problem horizon, as opposed to
the Fiedler vector heuristic which greedily chooses edges in
accordance with the direction of steepest ascent in $\lambda_2$.

\begin{figure}[h]
	\centering
	\includegraphics[scale=0.48]{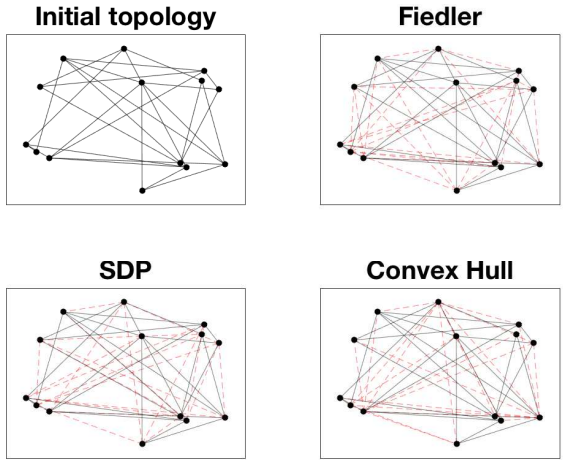}
	\caption{Initial topology of $14$ nodes and $28$ edges. Three
          methods are implemented to grow the network to $53$ edges,
          with the additional edges plotted as red dotted lines.}
	\label{fig:net-tops}
\end{figure}

\begin{figure}[h]
	\centering
	\includegraphics[scale=0.45]{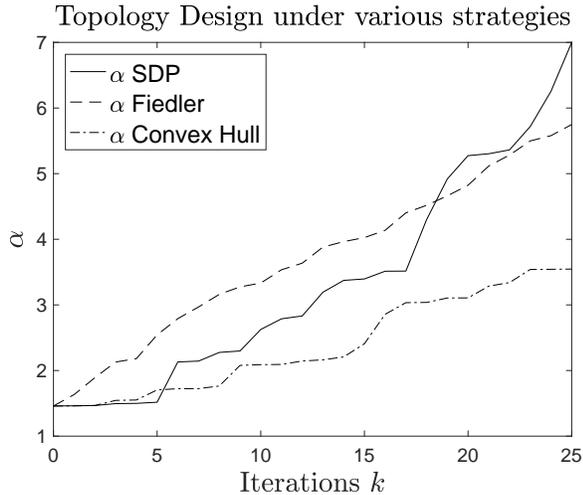}
	\caption{Performance each method over $k = 25$ iterations.}
	\label{fig:top-des}
\end{figure}

Next, we study a small network of $7$ nodes and $11$ edges so that the
solutions to $\Px 2$ and $\Px 3$ may be brute-forcibly computed to
test Algorithms~\ref{alg:rand}--\ref{alg:point}, with $\eta_1(k)$
being used for Algorithm~\ref{alg:cvx-comb}. 
We choose $\P = [0.25,0.75]^{11} \cap
\setdef{p}{\sum_l p_l \leq 4.25}$ and $\S = \{0,1\}^{11} \cap
\setdef{s}{\sum_l s_l \leq 5}$. We run the algorithms for $T = 30$
iterations and plot the results at each iteration in
Figure~\ref{fig:algs}.


\begin{figure}[h]
	\centering
	\includegraphics[scale=0.45]{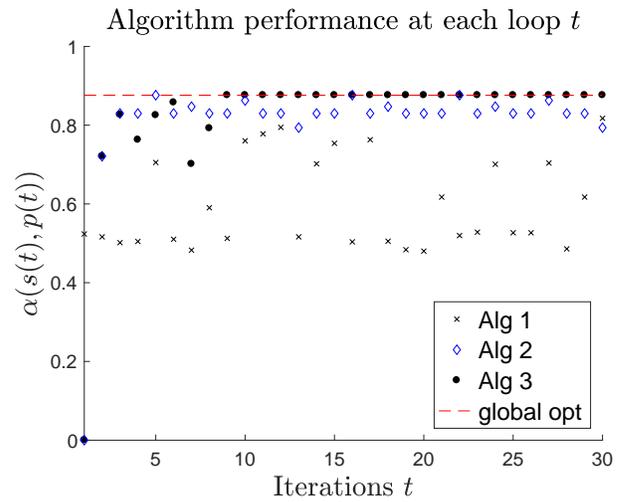}
	\caption{Performance of
          Algorithms~\ref{alg:rand}--\ref{alg:point} at each loop
          $t$.}
	\label{fig:algs}
\end{figure}

Clearly, Algorithm~\ref{alg:rand} does not improve across iterations
as it does not utilize information about the attacker's plays from
previous iterations, and it achieves a maximum value of
$\alpha(s(t^\star),p(t^\star)) = 0.8175$, which is a bit below the
global optimum $\alpha(s^\star,p^\star) =
0.8762$. Algorithm~\ref{alg:cvx-comb} achieves the global optimum
$\alpha(s(t^\star),p(t^\star)) = 0.8762$ the fastest, at $t^\star=5$,
although it never reaches this point again and instead oscillates
around suboptimal points, indicating that optimality may not be
reliably attained in general. 
Counter-intuitively, the performance of Algorithm~\ref{alg:point} does not improve
monotonically in $t$, although this should be expected: early solutions $s(t)$
may get ``lucky'', in some sense, but the subsequent iteration may allow
the attacker to jump to a new vulnerable part of the space. Once the
algorithm achieves the global optimum at $t^\star = 9$, it does not
dip below this for the remainder of the time horizon. Finally, we note that the behavior of each algorithm observed here is typical when implemented on other small graphs.

\section{Concluding Remarks}\label{sec:conclusion}

This paper introduced three related problems motivated by studying the algebraic connectivity of a graph by
adding edges to an initial topology or protecting edges under the case
of a disturbance or attack on the network. We developed a novel SDP relaxation to address the NP-hardness of the design and
demonstrated in simulation that it is superior to existing methods
which are greedy and cannot accommodate general constraints. In addition, we studied the
dynamics of the game that may be played between a network coordinator and
strategic attacker. We developed the notion of an
optimal preventive solution for the coordinator and proposed effective heuristics to find such a solution guided by characterizations of the solutions to the attacker's problem. Future work includes characterizing
the performance of our SDP relaxation and developing an algorithm which provably converges to the optimal
preventive strategy.


\bibliographystyle{abbrv}
\bibliography{alias,SMD-add}

\end{document}

%% file: macros-abbrev.tex
\newcommand{\until}[1]{\{1,\dots, #1\}}

\newcommand{\supscr}[2]{#1^{\textup{#2}}}
\newcommand{\setdef}[2]{\{#1 \; | \; #2\}}

\newcommand{\diag}[1]{\operatorname{diag}(#1)}

\newcommand{\spn}{\operatorname{span}}

\newcommand{\rank}{\operatorname{rank}}

\renewcommand{\epsilon}{\varepsilon}

\newcommand{\chull}[1]{\operatorname{co}(#1)}

\usepackage{psfrag}